\documentclass[12pt]{article}
\usepackage{latexsym, amsthm, amscd, graphicx, euscript, float, subfig}

\usepackage{enumerate,amsmath,amssymb}
\newcommand{\R}{\mathbb{R}}
\newcommand{\B}{\mathbb{B}}
\renewcommand{\H}{\mathbb{H}}
\newcommand{\C}{\mathbb{C}}
\newcommand{\Rn}{{\mathbb{R}^n}}
\newcommand{\Hn}{{\mathbb{H}^n}}
\newcommand{\Bn}{{\mathbb{B}^n}}
\newcommand{\arcsinh}{\,\textnormal{arsinh}\,}

\newcommand{\arctanh}{\,\textnormal{artanh}\,}

\newcommand{\comment}[1]{}

\newcounter{minutes}
\divide\time by 60
\newcounter{hours}
\multiply\time by 60
\addtocounter{minutes}{-\time}

\swapnumbers
\numberwithin{equation}{section}
\theoremstyle{plain}

\newtheorem{theorem}[equation]{Theorem}
\newtheorem{corollary}[equation]{Corollary}
\newtheorem{proposition}[equation]{Proposition}
\newtheorem{lemma}[equation]{Lemma}

\newtheorem{remark}[equation]{Remark}

\begin{document}

\begin{center}
{\bf \large Apollonian circles and hyperbolic geometry}
\end{center}

\vspace{0.5mm}

\begin{center}
{\bf \large  Riku Kl\'en and Matti Vuorinen}
\end{center}

\vspace{1mm}

\begin{center}{\small
\texttt{File:~\jobname.tex,
          printed: \number\year-\number\month-\number\day,
          \thehours.\ifnum\theminutes<10{0}\fi\theminutes} }
\end{center}

{\bf Abstract:}  The goal of this paper is to study two basic problems of hyperbolic geometry. The first problem is to compare the hyperbolic and Euclidean distances. The second problem is to find hyperbolic counterparts of some basic geometric constructions such as the construction of the middle point of a hyperbolic geodesic segment.
Apollonian circles have a key role in this study.

%\noindent
{\bf Keywords.} Apollonian circles, hyperbolic geometry
%\noindent

{\bf Mathematics Subject Classification 2010.} 51M10

%%%%%%%%%%%%%%%%%%%%%%%%%%%%%%%%%%%%%%%%%%%%%%%%
\section{Introduction}

The hyperbolic geometry was born about two centuries ago as a result of the independent work of J. Bolyai and N. Lobaschevsky \cite{m}. Its discovery solved the two millenia old problem about the role of the Parallel Postulate in Euclidean geometry: this postulate cannot be left out from the Euclidean geometry. Hyperbolic geometry provides an example of a geometry which satisfies all the postulates of the Euclidean geometry except that the Parallel Postulate does not hold. During the past two decades the hyperbolic geometry has surfaced in a number of contexts which do not belong to geometry proper: geometric function theory, discrete group theory, modern theory of quasiconformal and quasiregular mappings in $n$-dimensional Euclidean spaces and also in fields such as relativity theory \cite{u} and graphical art of Escher \cite{sc}. In function theoretic applications the hyperbolic metric is often more natural than the Euclidean metric. All these developments have lead to attempts to generalize hyperbolic metric to subdomains of the Euclidean space. For instance the quasihyperbolic metric of F.W.Gehring and B.P. Palka \cite{gp} and the Apollonian metric originally introduced by D. Barbilian \cite{barb} and rediscovered by A.F. Beardon \cite{b2} are two such hyperbolic type metrics, both studied in several recent PhD theses \cite{se}, \cite{ibr}, \cite{has}, \cite{lin}, \cite{sahoo}, \cite{klen}, \cite{manoj}. Quasihyperbolic metric as a tool of quasiconformal mapping theory is 
studied in \cite{va} and \cite{vu2}. The geometry defined by the quasihyperbolic metric in the context of Banach spaces is explored in \cite{krt,RasilaTalponen12}. For an interesting historical survey of the Apollonian metric see \cite{bs}.

Our goal is to keep the prerequisites for reading this paper as minimal as possible and, this in mind, we try to list carefully the necessary basic information in the introduction. We assume that the reader is familiar with basic facts about conformal mappings and M\"obius transformations (see \cite{ah}) of the complex plane. We occasionally also need some properties of M\"obius transformations of the M\"obius space $\overline{\R}^n = \Rn \cup \{ \infty \}$ and refer the reader to \cite{b1} or to \cite{ah2}. We often identify $\R^2$ with the complex plane $\C$. We use notation $B^n(x,r)$ and $S^{n-1}(x,r)$ for Euclidean balls and spheres, respectively. We abbreviate $\Bn = B^n(0,1)$.

%%%%%%%%%%%%%%%%%%%%%%%%%%
\subsection{The family of Apollonian circles}

For a fixed pair of points $x,y \in \Rn$, $x \neq y$, and $c > 0$ we define the Apollonian ball with base points $x,y$ by
\[
  B_{x,y}^c = \{ z \in \Rn \colon |x-z|<c|y-z| \}.
\]
Clearly $B_{x,y}^c$ is an open set with $x \in B_{x,y}^c$ and $\partial B_{x,y}^c = \partial B_{y,x}^{1/c}$. Apollonius' theorem says that $\partial B_{x,y}^c$ is a sphere which for $c=1$ reduces to the hyperplane through the midpoint $(x+y)/2$ of the segment $[x,y] = \{ z \in \Rn \colon z=tx+(1-t)y,\, 0 \le t \le 1 \}$ perpendicular to the segment $[x,y]$. Clearly $B_{x,y}^c$ and $B_{y,x}^c$ are symmetric with respect to this hyperplane and $B_{x,y}^{c_1} \subset B_{x,y}^{c_2}$ for $0<c_1<c_2$. This means that Apollonian balls with fixed base points are ordered by inclusion. The following lemma gives the Euclidean center point and the radius of an Apollonian circle.

\begin{lemma}\cite[p. 5, Exercise 1.1.25]{kr}\label{apollonian}
  Let $x,y \in \C$ and $c \in (0,1)$. Then $B_{x,y}^c = B^2(z,r)$ for
  \[
    z = \frac{y-c^2x}{1-c^2} \quad \textrm{and} \quad  r = \frac{c|x-y|}{1-c^2}.
  \]
\end{lemma}

%Apollonius circle \cite[18.3]{p}

%Apollonian circles also arise in electrostatics as field lines of two point charges \cite{some book}. Tarkista!!

%\emph{Apollonian circle} can be defined in four different ways
%\begin{enumerate}
%  \item the set of points whose distances from two fixed points are in constant ratio $c>0$,
%  \item a circle that is tangent to three given distinct triangles,
%  \item a circle passing through a vertex and both isodynamic points of a given triangle,
%  \item the circle that touches all three excircles of a triangle and encompasses them.
%\end{enumerate}

%%%%%%%%%%%%%%%%%%%%%%%%%%
\subsection{Cross ratio and absolute ratio}

For distinct $a,b,c,d \in \C$ we define the \emph{cross ratio} by
\[
  [a,b,c,d] = \frac{(a-c)(b-d)}{(a-b)(c-d)}
\]
and for distinct $a,b,c,d \in \Rn$ we define the \emph{absolute ratio} by
\[
  |a,b,c,d| = \frac{|a-c||b-d|}{|a-b||c-d|}.
\]
The cross ratio $[a,b,c,d]$ is a complex number and it is real if and only if the points $a,b,c,d$ are on the same circle. Both cross ratio and absolute ratio are invariant under M\"obius transformations. Moreover, a mapping $f \colon \overline{\R}^n \to \overline{\R}^n$ is M\"obius transformation if and only if it preserves absolute ratios (see \cite[3.2.7]{b1}). Observe that the absolute ratio depends on the order of points and e.g. $|a,b,c,d| \cdot |a,c,b,d| = 1$.

For a domain $G \subset \overline{\R}^n$, $\textrm{card} ( \overline{\Rn} \setminus G ) \ge 2$ and $x,y \in G$ the boundary $\partial G$ has at least two points $a_0,d_0$ such that
\[
  \alpha_G(x,y) := \sup_{a,d \in \partial G} \log |a,x,y,d| = \log \left( \sup_{a \in \partial G}  \frac{|a-y|}{|a-x|} \sup_{d \in \partial G}  \frac{|x-d|}{|y-d|} \right) = \log |a_0,x,y,d_0|.
\]
A simple verification shows that the quantity $\alpha_G$
satisfies the triangle inequality. It is called the Apollonian distance and
\[
  X = \sup_{a \in \partial G} \frac{|a-y|}{|a-x|}, \quad Y = \sup_{d \in \partial G} \frac{|x-d|}{|y-d|}
\]
are the Apollonian parameters. For given $x,y \in G$ we have
\[
  B_{x,y}^Y = \left\{ z \in \overline{\Rn} \colon \frac{|z-x|}{|z-y|} < Y \right\}, \quad B_{y,x}^YX= \left\{ z \in \overline{\Rn} \colon \frac{|z-y|}{|z-x|} < X \right\}.
\]
The Apollonian distance maximizes the size of Apollonian balls, $\alpha_G(x,y) = XY$.
The Apollonian distance defines a metric, whenever the complement of the domain $G$
is not contained in a hyperplane \cite{b2}.

%%%%%%%%%%%%%%%%%%%%%%%%%%
\subsection{Hyperbolic distance}

For a domain $G \subsetneq \Rn$, $n \ge 2$ and a continuous function $w \colon G \to (0,\infty)$ we define the \emph{$w$-length} of a rectifiable arc $\gamma \subset G$ by
\[
  \ell_w(\gamma) = \int_{\gamma}w(z)|dz|,
\]
and the \emph{$w$-distance} by
\begin{equation}\label{wmetric}
  m_w(x,y) = \inf_\gamma \ell_w(\gamma),
\end{equation}
where the infimum is taken over all rectifiable curves in $G$ joining $x$ and $y$. We say that a curve $\gamma \colon [0,1] \to G$ is a \emph{geodesic segment} if for all $t \in (0,1)$ we have
\[
  m_w(\gamma(0),\gamma(t)))+m_w(\gamma(t),\gamma(1)))=m_w(\gamma(0),\gamma(1))).
\]

The \emph{hyperbolic distance} in $\Hn$ is defined by the weight function $w_\Hn(z)=1/z_n$ and in $\Bn$ by the weight function $w_\Bn(z) = 2/(1-|z|^2)$. By \cite[p. 35]{b1} we have
\begin{equation}\label{hyperbolicH}
  \cosh \rho_\Hn(x,y) = 1+\frac{|x-y|^2}{2 x_n y_n}
\end{equation}
for all $x,y \in \Hn$ and by \cite[p. 40]{b1} we have
\begin{equation}\label{hyperbolicB}
  \sinh \frac{\rho_\Bn(x,y)}{2} = \frac{|x-y|}{\sqrt{1-|x|^2}\sqrt{1-|y|^2}}
\end{equation}
for all $x,y \in \Bn$. With the respective weight functions given above the definitions \eqref{hyperbolicH} and \eqref{hyperbolicB} coincide with \eqref{wmetric}. If the domain is understood from the context we use notation $\rho$ instead of $\rho_\Hn$ and $\rho_\Bn$. The hyperbolic distance can equivalently be defined for $G \in \{ \Bn,\Hn \}$ as
\begin{equation}\label{rhoabsratio}
  \rho_G (x,y) = \sup \{ \log |a,x,y,b|  \colon a,b \in \partial G \} = \log |x',x,y,y'|,
\end{equation}
where $x',y' \in \partial G$ such that the circle that contains $x,x',y,y'$ is orthogonal to $\partial G$ and the points $x',x,y,y'$ occur in this order. In particular, \eqref{rhoabsratio} says that for $G \in \{ \Bn,\Hn \}$ the Apollonian distance agrees with the hyperbolic distance $\rho_G = \alpha_G$.

Hyperbolic geodesics are arcs of circles that are orthogonal to the boundary of the domain. More precisely, a hyperbolic geodesic segment is the intersection of the domain with the Euclidean circle or straight line which is orthogonal to the boundary of the domain, see \cite{b1}. Therefore, the points $x'$ and $y'$ are the end points of the hyperbolic geodesic segment which contains $x$ and $y$. For any two distinct points the hyperbolic geodesic segment is unique.

Given two distinct points $x,y \in \H^2$ the circle $C_{xy}$ containing $x$, $y$ and perpendicular to the $x$-axis can be characterized as the circle through the three points $x$, $y$, $\overline{x}$, where $\overline{x}$ is the image of $x$ under the reflection in the $x$-axis (the map $(u,v) \mapsto (u,-v)$). Moreover, the points $x'$, $x$, $y$, $y'$ occur in this order on $C_{xy}$ and $\{ x',y' \} = C_{xy} \cap \R$.

Similarly, for $x,y \in \B^2$ the circle $C_{xy}$ containing $x,y$ and perpendicular
to $\partial \B^2$ is the circle through $x$, $y$, $x^*$, where $x^*=x/|x|^2$
is the image of $x$ under the reflection in the unit circle.
Again $\{ x',y' \} = C_{xy} \cap \B^2$. Hyperbolic distance is invariant
under M\"obius transformations of $\Bn$ onto $\Bn$ or onto $\Hn$.

\begin{figure}[ht!]
  \begin{center}
    \includegraphics[width=6cm]{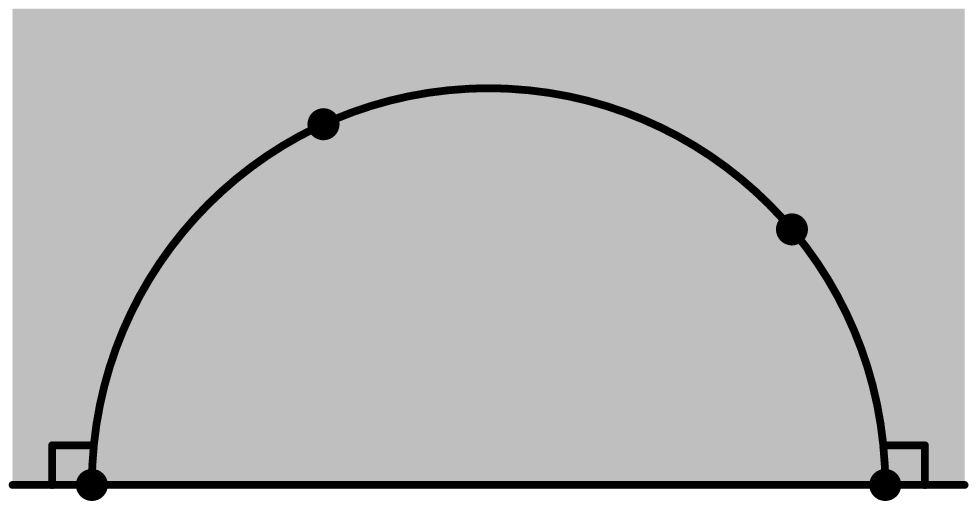}\hspace{5mm}
    \includegraphics[width=4cm]{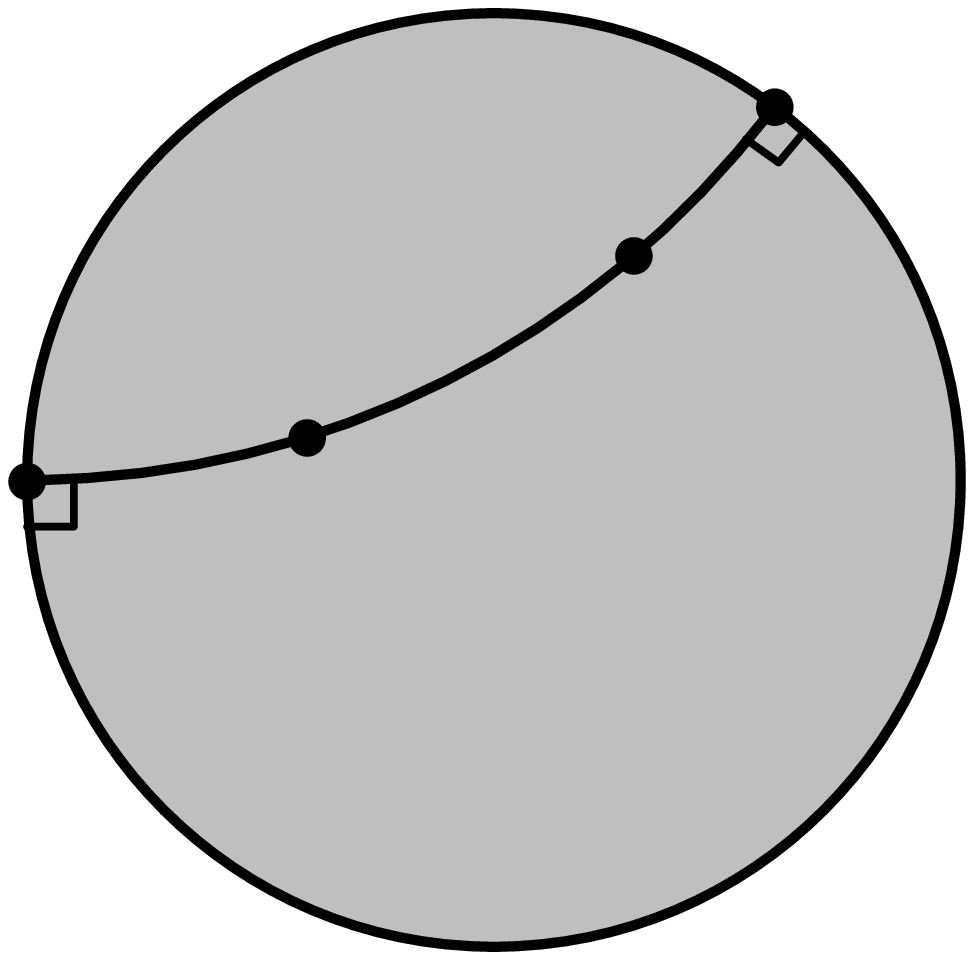}
    \caption{An example of a hyperbolic geodesic segment in the half plane and the unit disk.}
  \end{center}
\end{figure}

The above basic facts can be found in our standard references \cite{b1,ah2} and
in many other sources on hyperbolic geometry such as \cite{a,bm,kl}.
Farreaching and specialized advanced texts discussing hyperbolic geometry
include \cite{r, mar}.

Our plan here is to show the many links between Apollonian circles and hyperbolic geometry.
In particular, we give an interpretation to geodesics in terms of Apollonian circles and
give examples for the determination of some natural concepts of hyperbolic geometry such
as the midpoint of a geodesic and the "base points" of a geodesic. We also discuss
the natural question of comparing distances in the Euclidean and hyperbolic geometry
recently investigated by C.J. Earle and L.A. Harris \cite{eh}.

%%%%%%%%%%%%%%%%%%%%%%%%%%%%%%%%%%%%%%%%%%%%%%%%
\section{Hyperbolic geometry in the unit disk}

We denote Euclidean balls and spheres by $B^n(x,r)$ and $S^{n-1}(x,r)$, respectively. For any metric $m$ we denote metric ball
\[
  B_m(x,r) = \{ y \colon m(x,y) < r \}.
\]

\begin{proposition}
  The hyperbolic sphere $\partial B_\rho(x,r)$, $x \in \Bn$, $r > 0$, is an Apollonian sphere with base points $x$ and $x^* = x/|x|^2$.
\end{proposition}
\begin{proof}
  Fix $y \in \partial B_\rho (x,r)$. Then it follows from \eqref{hyperbolicB} that
  \[
    \tanh^2 \frac{\rho(x,y)}{2} = \frac{|x-y|^2}{A[x,y]^2} = \tanh^2 \frac{r}{2} = c(r)^2,
  \]
  where $A[x,y]^2 = |x-y|^2+(1-|x|^2)(1-|y|^2)$, $x,y \in \Bn$, is the Ahlfors bracket.
  On the other hand, a simple verification shows that $|x||x^* -y|= A[x,y]$ and hence
  \[
    \frac{|x-y|^2}{|x^*-y|^2} = \frac{|x-y|^2}{(A[x,y]/|x|)^2} = |x|^2 \left( \frac{|x-y|}{A[x,y]} \right)^2 = |x|^2 c(r)^2.
  \]
  This is independent of $y$ and hence $\partial B_\rho (x,r)$ is an Apollonian sphere. We also see that
  \[
    B_\rho(x,r) = B_{x,x^*}^{|x|c(r)}.
  \]
\end{proof}

From this proof we can read off the following simple formula for $x,y \in \Bn$
\[
    \tanh \frac{\rho(x,y)}{2} = \frac{|x-y|}{|x||x^* -y|} \,, \,\,\, x^* = x/|x|^2 \,.
  \]

Some basic properties of orthogonal circles will be recalled now.
For that purpose the reader might wish to see \cite[p. 6, Exercise 1.1.27]{kr}.
The next result gives a formula for hyperbolic geodesic segment in $\B^2$.

\begin{lemma}\label{orthogB2}
  Let $a \in \C$ with $|a| > 1$. Then $S^1(a,r)$ is orthogonal to $S^1(0,1)$ for $|a|^2 = 1+r^2$. Given $x,y \in \B^2$ such that $0$, $x$ and $y$ are noncollinear the orthogonal circle $S^1(a,r)$ contains $x$ and $y$
 if  \[
    r=\frac{|x-y||x|y|^2-y|}{2|y||x_1 y_2-x_2 y_1|} \quad \textrm{and} \quad a=i\frac{y(1+|x|^2)-x(1+|y|^2)}{2(x_2 y_1-x_1 y_2)}
  \]
  and $S^1(a,r) \cap S^1(0,1) = \{ z \in \C \colon z=a/|a| \exp(\pm i\theta),\, \theta = \arccos (1/|a|) \}$.
\end{lemma}

\begin{lemma}\label{orthoB}
  Let $x \in \Rn$, $r>0$ and $y,z \in S^{n-1}(x,r)$. Then $y,z \in S^{n-1}(w,|y-w|)$ and $S^{n-1}(w,|y-w|)$ is orthogonal to $S^{n-1}(x,r)$, where
  \[
    w = x+\frac{|y-x|^2}{\left| (y+z)/2-x \right|^2} \left( \frac{y+z}{2}-x \right).
  \]
\end{lemma}
\begin{proof}
  Let us denote $s=(y+z)/2$. We have
  \[
    w-x = \lambda \left( s-x \right)
  \]
  for a scalar $\lambda > 1$. Because triangles $xsy$ and $xyw$ are right and similar, we obtain
  \[
    \frac{|x-y|}{|x-w|} = \frac{|x-s|}{|x-y|}.
  \]
  By the above equalities we obtain
  \[
    w = x+\frac{|x-y|^2}{|x-s|^2}(s-x)
  \]
  and the assertion follows.
\end{proof}

\begin{proposition}\label{distancefromline}
  The Euclidean distance of the line containing distinct points $a,b \in \Rn$ to the origin is
  \[
    \frac{\sqrt{(|a-b|^2-(|a|-|b|)^2)((|a|+|b|)^2-|a-b|^2)}}{2|a-b|}.
  \]
\end{proposition}
\begin{proof}
  We denote the line that contains $a$ and $b$ by $l$, the angle $\measuredangle (a,0,b)$ by $\alpha$ and area of triangle $0ab$ by $A$. Since $2A = |a||b|\sin \alpha = d(l,0) |a-b|$ we obtain that
  \begin{equation}\label{distancetoline}
    d(l,0) = \frac{|a||b|}{|a-b|}\sin \alpha = \frac{\sqrt{|a|^2|b|^2-(a \cdot b)^2}}{|a-b|}.
  \end{equation}
  By the law of Cosines and the fact that $\sin^2 \alpha =1-\cos^2 \alpha$ we obtain
  \[
    \sin \alpha = \frac{\sqrt{(|a-b|^2-(|a|-|b|)^2)((|a|+|b|)^2-|a-b|^2)}}{2|a||b|}
  \]
  which together with \eqref{distancetoline} implies the assertion.
\end{proof}

%\begin{proof}[Proof of Lemma \ref{apollonian}]
%  Let us denote $C = S^1(z,r)$. Now $C$ is a circle with center point on the line $l$ that contains $x$ and $y$. We denote intersection points of $l$ and $C$ by $z_1$ and $z_2$ so that $z_1 \in [x,y]$.Now $|x-y| = |x-z_1|+|y-z_1| = (1+1/c)|x-z_1|$ which is equivalent to $|x-z_1| = c|x-y|/(1+c)$ implying
%  \[
%    z_1 = x+(y-x)|x-z_1|/|x-y| = \frac{1}{1+c}x+\frac{c}{1+c}y = \frac{x+cy}{1+c}.
%  \]
%  Similarly $|x-y| = |y-z_2|-|x-z_2| = (1/c-1)|x-z_2|$ which is equivalent to $|x-z_2| = c|x-y|/(1-c)$ implying
%  \[
%    z_2 = y+(x-y)(|x-y|+|x-z_2|)/|x-y| = \frac{1}{1-c}x-\frac{c}{1-c} = \frac{x-cy}{1-c}.
%  \]
%  We obtain
%  \[
%    c = \frac{z_1+z_2}{2} = \frac{1}{2} \left( \frac{x+cy}{1+c}+\frac{x-cy}{1-c} \right) = \frac{x-c^2y}{1-c^2}
%  \]
%  and
%  \[
%    r = \frac{|x-z_1|+|x-z_2|}{2} = \frac{1}{2} \left( \frac{c|x-y|}{(1+c)} + \frac{c|x-y|}{(1-c)} \right) = \frac{c|x-y|}{1-c^2}.
%  \]
%\end{proof}

Our aim is to find the hyperbolic midpoint $z$ of points $x,y \in \B^2$. By \eqref{hyperbolicB} it is clear that $\rho(x,z) = \rho(y,z)$ is equivalent to
\begin{equation}\label{apollonianratio}
  |x-z| = \sqrt{\frac{1-|x|^2}{1-|y|^2}}|y-z|
\end{equation}
and the next result characterizes the points satisfying this equality.

\begin{corollary}\label{circlea2}
  Let $x,y \in \B^2$ with $|y| < |x|$. Then
  \begin{eqnarray*}
    C & = & \{ z \in \R^2 \colon \rho(x,z)=\rho(y,z) \}\\
    & = & \left\{ z \in \R^2 \colon \frac{|x-z|}{|y-z|} = A,\, A=\sqrt{\frac{1-|x|^2}{1-|y|^2}} \right\} = S^1(w,r),
  \end{eqnarray*}
  where
  \[
    w = \frac{x-A^2y}{1-A^2} \quad \textrm{and} \quad  r = \frac{A|x-y|}{1-A^2}.
  \]
  Moreover, the circle $C$ is orthogonal to $S^1(0,1)$ and the geodesic $J$ joining $x$ and $y$.
\end{corollary}
\begin{proof}
  Since $|y| < |x|$ we have $A \in (0,1)$ and values of $w$ and $r$ follow from Lemma \ref{apollonian}.

  We prove next orthogonality to $S^1(0,1) \,$. Since
  \begin{eqnarray*}
    & & |w|^2 = 1+r^2\\
    & \Longleftrightarrow & |x|^2 + A^4 |y|^2-2A^2 x \cdot y = (1-A^2)^2+A^2(|x|^2+|y|^2-2 x \cdot y)\\
    & \Longleftrightarrow & (1-A^2)|x|^2-A^2(1-A^2)|y|^2 = (1-A^2)^2\\
    & \Longleftrightarrow & |x|^2-A^2|y|^2 = 1-A^2\\
    & \Longleftrightarrow & \frac{|x|^2-|y|^2}{1-|y|^2} = \frac{|x|^2-|y|^2}{1-|y|^2}
  \end{eqnarray*}
  the circle $C$ is orthogonal to $S^1(0,1)$ by Lemma \ref{orthogB2}.

  Finally, we prove orthogonality of $S^1(w,r)$ and $J$. Let us denote intersection of $S^1(w,r)$ and $J$ in $\B^2$ by $z$. Let $T$ be a sense-preserving M\"obius mapping with $T(z) = 0$ (see \cite[1.34]{vu}). Since $\rho(x,z) = \rho(y,z)$ we have $\rho(T(x),0) = \rho(T(y),0)$. Let $v \in C \cap \B^2$. By definition of $C$ we have $\rho(x,v) = \rho(y,v)$ and thus $\rho(T(x),T(v)) = \rho(T(y),T(v))$ implying that $T(C)$ is orthogonal to $[T(x),T(y)]$. Since $T$ is M\"obius we have that $S^1(w,r)$ is orthogonal to $J$.
\end{proof}

\begin{lemma}\label{hypmidpoint}
  The hyperbolic midpoint of points $x,y \in \B^2$ is $z$ with
  \[
    |z| = \frac{\sqrt{(|a_1-a_2|^2-(|a_1|-|a_2|)^2)((|a_1|+|a_2|)^2-|a_1-a_2|^2)}}{2|a_1-a_2|}-\frac{r_1 r_2}{\sqrt{r_1^2+r_2^2}}
  \]
  and
  \[
    \frac{z}{|z|} = \pm \sqrt{\frac{a_1-a_2}{\overline{a_2}-\overline{a_1}}},
  \]
  where $a_1$ is the center of the hyperbolic geodesic joining $x$ and $y$, $r_1 =|a_1-x|$ and $(a_2,r_2)$ is $(w,r)$ of Corollary \ref{circlea2}.
\end{lemma}
\begin{proof}
  By Corollary \ref{circlea2} circles $S^1(0,1)$,  $S^1(a_1,r_1)$ and $S^1(a_2,r_2)$ are pairwise orthogonal. Therefore by Lemma \ref{orthogB2} $|a_i|^2=r_1^2+1$ for $i=1,2$ and $|a_i-z|^2 = r_i^2$ implies $|z|^2-\overline{z}a_i-z \overline{a_i}+1=0$, which is equivalent to
  \[
    \frac{z}{|z|} = \pm \sqrt{\frac{a_1-a_2}{\overline{a_2}-\overline{a_1}}}.
  \]
  By the orthogonality of $S^1(a_1,r_1)$ and $S^1(a_2,r_2)$ and similar right triangles $a_1 a_2 z$ and $z v a_2$, $v \in [a_1,a_2]$, we have $|z-v| =r_1 r_2 /\sqrt{r_1^2+r_2^2}$. Since $z \in [0,v]$ we have by Proposition \ref{distancefromline}
  \begin{eqnarray*}
    |z| & = & |v|-|z-v|\\
    & = & \frac{\sqrt{(|a_1-a_2|^2-(|a_1|-|a_2|^2))((|a_1|+|a_2|)^2-|a_1-a_2|^2)}}{2|a_1-a_2|}-\frac{r_1 r_2}{\sqrt{r_1^2+r_2^2}}.
  \end{eqnarray*}
\end{proof}

\begin{figure}[ht!]
  \begin{center}
    \includegraphics[width=5cm]{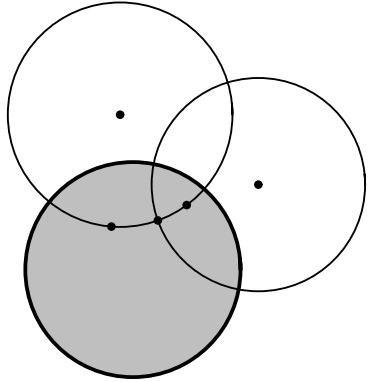}
    \caption{The hyperbolic midpoint in the unit disk.}
  \end{center}
\end{figure}

Lemma \ref{hypmidpoint} gives the hyperbolic midpoint of two points in $\B^2$. We will now describe a geometric construction for finding the hyperbolic midpoint.

\setcounter{subsection}{\value{equation}}
%%%%%%%%%%%%%%%%%%%%%%%%%%
\subsection{Bisection of hyperbolic segment}

Let $x,y \in \B^2$. By \eqref{apollonianratio} points $z \in \B^2$ such that $\rho(x,z) = \rho(y,z)$ are on an Apollonian circle $\partial B_{x,y}^c$, where $c$ depends only on $|x|$ and $|y|$. The hyperbolic geodesic joining $x$ and $y$ is an arc of a circle $S^1(a,r)$, where $a$ and $r$ depend only on $x$ and $y$. Formulas for $a$ and $r$ are given in Lemma \ref{orthogB2}.
\begin{enumerate}
  \item Construct the Apollonian circle $\partial B_{x,y}^c$, which contains the points $z$ such that $\rho(x,z) = \rho(y,z)$.
  \item Construct the circle $S^1(a,r)$, which contains the hyperbolic geodesic joining $x$ and $y$.
  \item Hyperbolic midpoint $z$ of points $x$ and $y$ is the intersection of these two circles, $z = \partial B_{x,y}^c \cap S^1(a,r) \cap \B^2$.
\end{enumerate}

In the particular case $y=0$ the midpoint $z$ can be found in a very simple way.
In fact, for a fixed $x \in \B^2 \cap [0,1), $ the hyperbolic midpoint of $[0,x]$ is the
point of intersection of the segments $[0,x]$ and $[-i,x+ i \sqrt{1- |x|^2}]\,$ by
\cite[1.41(2)]{vu}. Bisection of hyperbolic segment is also considered in \cite{vw}.

%%%%%%%%%%%%%%%%%%%%%%%%%%%%%%%%%%%%%%%%%%%%%%%%
\section{Hyperbolic distance in the unit ball}

For the sequel it is convenient to interpret hyperbolic balls in terms of Euclidean geometry as follows: for $x \in \Bn$, $r > 0$
\begin{equation}\label{hypeucl}
  B_\rho(x,r) = B^n \left( w,s \right), \quad w=\frac{x(1-t^2)}{1-|x|^2 t^2}, \quad s= \frac{(1-|x|^2)t}{1-|x|^2 t^2},
\end{equation}
where $t = \tanh (r/2)$. This is a well-known formula, see e.g. \cite[(2.22)]{vu}.

We shall use the following basic inequality for $x \in [0,1]$
\begin{equation}\label{basicineq}
  \sqrt{1-x} \le 1-\frac{x}{2}.
\end{equation}

\begin{lemma}\label{upperbounds}
  For $r,s \in [0,1)$ we have
  \begin{enumerate}
    \item[(1)] $\sqrt{(1-r^2)(1-s^2)} \le \displaystyle 1-rs-\frac{1}{2}\frac{(r-s)^2}{1-rs} \le \displaystyle 1-\left( \frac{r+s}{2} \right)^2$,
    \item[(2)] $\sqrt{(1-r^2)(1-s^2)} \le \displaystyle \sqrt{1+r^2s^2}-\frac{r^2+s^2}{2\sqrt{1+r^2s^2}}$,
    \item[(3)] $\sqrt{(1-r^2)(1-s^2)} \le \displaystyle 1+rs-\frac{1}{2}\frac{(r+s)^2}{1+rs}$,
%    \item[(2')] $\displaystyle 1-rs-\frac{1}{2}\frac{(r-s)^2}{1-rs} \le \displaystyle \sqrt{1+r^2s^2}-\frac{r^2+s^2}{2\sqrt{1+r^2s^2}}$,
%    \item[(3')] $\displaystyle \sqrt{1+r^2s^2}-\frac{r^2+s^2}{2\sqrt{1+r^2s^2}} \le \displaystyle 1+rs-\frac{1}{2}\frac{(r+s)^2}{1+rs}$.
  \end{enumerate}
\end{lemma}
\begin{proof}
    \textit{(1)} The first inequality holds because by \eqref{basicineq}
    \begin{eqnarray*}
      \sqrt{(1-r^2)(1-s^2)} & = & (1-rs) \sqrt{1-\left( \frac{r-s}{1-rs} \right)^2}\\
      & < & (1-rs) \left( 1-\frac{1}{2} \left( \frac{r-s}{1-rs} \right)^2 \right)\\
      & = & 1-rs-\frac{1}{2}\frac{(r-s)^2}{1-rs}.
    \end{eqnarray*}
    The second inequality is equivalent to
    \[
      \frac{(r-s)^2(1+rs)}{4(1-rs)} \ge 0
    \]
    and thus the assertion follows.

    \noindent \textit{(2)} The inequality holds because by \eqref{basicineq}
    \begin{eqnarray*}
      \sqrt{(1-r^2)(1-s^2)} & = & \sqrt{1+r^2s^2-(r^2+s^2)}\\
      & = & \sqrt{1+r^2s^2}\sqrt{1-\frac{r^2+s^2}{1+r^2s^2}}\\
      & < & \sqrt{1+r^2s^2} \left( 1-\frac{1}{2}\frac{r^2+s^2}{1+r^2s^2} \right) \\
      & = & \sqrt{1+r^2s^2}-\frac{r^2+s^2}{2\sqrt{1+r^2s^2}}
    \end{eqnarray*}

    \noindent \textit{(3)} The inequality holds because by \eqref{basicineq}
    \begin{eqnarray*}
      \sqrt{(1-r^2)(1-s^2)} & = & (1+rs) \sqrt{1-\left( \frac{r+s}{1+rs} \right)^2}\\
      & < & (1+rs) \left( 1-\frac{1}{2} \left( \frac{r+s}{1+rs} \right)^2 \right)\\
      & = & 1+rs-\frac{1}{2}\frac{(r+s)^2}{1+rs}.
    \end{eqnarray*}

%    PROVE (2') AND (3')!!!
\end{proof}

By using the previous lemma we find lower bounds for the hyperbolic distance in terms of the Euclidean distance.

\begin{theorem}\label{lowerboundsinBn}
  For $x,y \in \Bn$ we have
  \begin{enumerate}
    \item[(1)] $\sinh\displaystyle \frac{\rho(x,y)}{2} \ge \frac{|x-y|}{\sqrt{\displaystyle 1+\frac{|x|^4+|y|^4}{2}-|x|^2-|y|^2}}$,
    \item[(2)] $\tanh \displaystyle \frac{\rho(x,y)}{4} \ge \displaystyle \frac{|x-y|}{1+|x||y|+\sqrt{1-|x|^2}\sqrt{1-|y|^2}} \ge \frac{|x-y|}{2}$,
    \item[(3)] $\tanh \displaystyle \frac{\rho(x,y)}{4} \ge \displaystyle \frac{|x-y|}{2-\left( \displaystyle \frac{|x-y|}{2} \right)^2}$,
    \item[(4)] $\tanh \displaystyle \frac{\rho(x,y)}{4} \ge \displaystyle \frac{|x-y|}{2- \displaystyle \frac{(|x|-|y|)^2}{2}} \ge \displaystyle \frac{|x-y|}{2- \displaystyle \frac{1}{2}\frac{(|x|-|y|)^2}{1-|x||y|}}$,
    \item[(5)] $\tanh \displaystyle \frac{\rho(x,y)}{4} \ge \displaystyle \frac{|x-y|}{1+ |x||y|+\sqrt{1+|x|^2|y|^2}- \displaystyle \frac{|x|^2+|y|^2}{2\sqrt{1+|x|^2|y|^2}}}$,
    \item[(6)] $\tanh \displaystyle \frac{\rho(x,y)}{4} \ge \displaystyle \frac{|x-y|}{2+2|x||y|- \displaystyle \frac{1}{2}\frac{(|x|+|y|)^2}{1+|x||y|}}$,
    \item[(7)] $\tanh \displaystyle \frac{\rho(x,y)}{4} \ge \displaystyle \frac{|x-y|}{\sqrt{|x-y|^2+4\sqrt{1-|x|^2}\sqrt{1-|y|^2}}}$.
  \end{enumerate}
\end{theorem}
\begin{proof}
  (1) The assertion follows since $(1-|x|^2)(1-|y|^2) \le ((1-|x|^2)^2+(1-|y|^2)^2)/2 = 1+(|x|^4+|y|^4)/2-|x|^2-|y|^2$.

  \noindent (2) Follows from \cite[(2.27)]{vu} and \cite[(2.52 (3))]{vu}.

  \noindent (3)--(6). By \cite[(2.52 (3))]{vu}
  \[
    |x-y| \le (1+|x||y|+\sqrt{1-|x|^2}\sqrt{1-|y|^2} ) \tanh \frac{\rho(x,y)}{4}
  \]
  and the assertion follows from Lemma \ref{upperbounds}. For example
  \begin{eqnarray}
    |x-y| & \le & (1+|x||y|+\sqrt{1-|x|^2}\sqrt{1-|y|^2} ) \tanh \frac{\rho(x,y)}{4}\\
    & \le & \left( 1+|x||y|+1-\left( \frac{|x|+|y|}{2} \right)^2 \right) \tanh \frac{\rho(x,y)}{4}\\
    & = & \left( 2-\left( \frac{|x|-|y|}{2} \right)^2 \right) \tanh \frac{\rho(x,y)}{4}.
  \end{eqnarray}
  The first inequality of (4) follows from \cite[7.64 (24)]{avv}.

  \noindent (7) Follows from \cite[7.64 (25)]{avv}.
\end{proof}

\begin{remark}
  In Theorem \ref{lowerboundsinBn} are given various lower bounds for $\tanh (\rho(x,y)/4)$. Let us denote the better lower bound of Theorem \ref{lowerboundsinBn} (2) by $c_2$ and lower bounds of Theorem \ref{lowerboundsinBn} (3), (5) and (6) respectively by $c_3$, $c_5$ and $c_6$. It is possible to show that for all $x,y \in \Bn$
  \[
    c_6 \le c_5 \le c_3 \le c_2.
  \]
\end{remark}

Next we prove more lower bounds for the hyperbolic distance in $\Bn$.

\begin{lemma}\label{rhobound}
  Let $x,y \in \Bn$. Then
  \begin{eqnarray*}
    \rho(x,y) \ge \rho(x',y') & = & 2 \arcsinh \frac{|x-y|}{2 \left( {\sqrt{1+r^2}\sqrt{\displaystyle r^2-\frac{|x-y|^2}{4}}-r^2} \right) }\\
    & = & 2 \arctanh \left( (r+\sqrt{1+r^2})\tan \frac{\theta}{2} \right),
  \end{eqnarray*}
  where $x'$ and $y'$ are on the same geodesic joining $x$ and $y$ with $|x'-y'|=|x-y|$, $|x'|=|y'|$, $\theta = \measuredangle (0,a,x')$, $a \in \Rn$ and $r \ge 0$ such that the geodesic joining $x$ and $y$ is a subset of $S^{n-1}(a,r)$
\end{lemma}
\begin{proof}
  We prove first the inequality. Denote $2\alpha = \measuredangle (x_*,a,x)$, $2\beta = \measuredangle (x,a,y)$ and $2\gamma = \measuredangle (y_*,a,y)$, see Figure \ref{angles}. Now $\alpha+\beta+\gamma$ is constant and by \cite[(7.26)]{b1} and trigonometry
  \begin{eqnarray*}
    \exp \rho(x,y) & = & |x_*,x,y,y_*| = \frac{|x_*-y||x-y_*|}{|x_*-x||y-y_*|}\\
    & = & \frac{\sin(\alpha+\beta)}{\sin \alpha}\frac{\sin(\beta+\gamma)}{\sin \gamma}\\
    & = & \left( \frac{\sin \alpha \cos \beta+\cos \alpha \sin \beta}{\sin \alpha} \right) \left( \frac{\sin \beta \cos \gamma+\cos \beta \sin \gamma}{\sin \gamma} \right)\\
    & = & \left( \cos \beta + \frac{\sin \beta}{\tan \alpha} \right) \left( \cos \beta + \frac{\sin \beta}{\tan \gamma} \right).
  \end{eqnarray*}
  We fix $|x-y|$ and thus also $\beta$. Now $\gamma = c-\alpha$ for some positive constant $c < \pi/2$ and we consider the function
  \[
    f(\alpha) = \left( B + \frac{A}{\tan \alpha} \right) \left( B + \frac{A}{\tan (c-\alpha)} \right)
  \]
  for $\alpha \in (0,c)$, where $A = \sin \beta$ and $B = \cos \beta$. Clearly $f(\alpha) \to \infty$ as $\alpha \to 0$ or $\alpha \to c$. Since
  \[
    f'(\alpha) = \frac{A(B\sin c+A\cos c)}{\sin^2 \alpha \sin^2 (\alpha-c)}\sin (2\alpha-c)
  \]
  the function $f(\alpha)$ obtains its minimum at $\alpha = c/2$, which is equivalent to $\gamma = \alpha$. In other words, we obtain $\rho(x,y) \ge \rho(x',y')$.

  We prove next the first formula for $\rho(x',y')$. By selection of $x'$ and $y'$ it is clear that $\rho(x',y') = 2 \arcsinh (|x-y|/(1-|x'|^2))$ and thus we find $|x'|^2$. By the Pythagorean theorem we obtain that $|x'|^2 = (|a|-d)^2+|x-y|^2/4$ for $d = \sqrt{r^2-|x-y|^2/4}$. Since $|a| = \sqrt{1+r^2}$ we obtain
  \[
    |x'|^2 = 1+2 \left( r^2-\sqrt{r^2-\frac{|x-y|^2}{4}}\sqrt{1+r^2} \right).
  \]

  We prove finally the second formula for $\rho(x',y')$. Let $z$ be a point on the same geodesic with $x$ and $y$. For $\alpha = \measuredangle (0,a,z)$ we have
  \begin{eqnarray*}
    |z|^2 & = & (|a|-r \cos \alpha)^2+(r \sin \alpha)^2 = 1+r^2-2r \sqrt{1+r^2} \cos \alpha +r^2\\
    & = & 1+2r(r-\sqrt{1+r^2}\cos \alpha).
  \end{eqnarray*}
  Now
  \begin{eqnarray*}
    \rho(x',y') & = & 2 \int_0^\theta \frac{2r d\alpha}{1-|z|^2} = 2 \int_0^\theta \frac{2r d\alpha}{2r(\sqrt{1+r^2}\cos \alpha-r)}\\
    & = & 4 \arctanh \left( (r+\sqrt{1+r^2}) \tan \frac{\theta}{2} \right)
  \end{eqnarray*}
  and the assertion follows.
\end{proof}

\begin{figure}[ht!]
  \begin{center}
    \includegraphics[width=6cm]{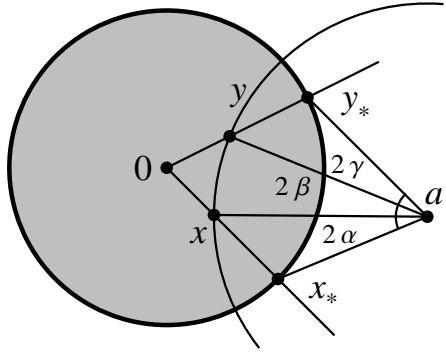}
    \caption{Angles $\alpha$, $\beta$ and $\gamma$ in the proof of Lemma \ref{rhobound}.\label{angles}}
  \end{center}
\end{figure}

\begin{lemma}
  For $x,y \in \Bn$ and $a$ as in Lemma \ref{orthogB2} we have
  \[
    \tanh \frac{\rho(x,y)}{4} \ge \frac{4-4|w|^2+|x-y|^2-\sqrt{(4-4|w|^2+|x-y|^2)^2-16|x-y|^2}}{4|x-y|},
  \]
  where
  \[
    w = a+\frac{|x-a|^2}{\left| (x+y)/2-a \right|^2} \left( \frac{x+y}{2}-a \right).
  \]
\end{lemma}
\begin{proof}
  The smallest hyperbolic ball $B_\rho$ with $x,y \in \partial B_\rho$ has the center point of the hyperbolic center point of $x$ and $y$. Thus by Lemma \ref{orthoB} $B_\rho = B^n(w,|x-w|)$. By \eqref{hypeucl}
  \[
    t=|x-w|=\frac{1+T^2+\sqrt{1+T^4+T^2(4|w|^2-2)}}{2T},
  \]
  where $T=\tanh ( \rho(x,y)/4 )$. Since $|x-y| \le 2t$ we obtain
  \[
    \frac{4-4|w|^2+|x-y|^2-\sqrt{(4-4|w|^2+|x-y|^2)^2-16|x-y|^2}}{4|x-y|} \le T
  \]
  which implies the assertion.
\end{proof}

\begin{lemma}
  For $x,y \in \Bn$ we have
  \[
    \tanh \frac{\rho(x,y)}{2} \ge \frac{|z|^2-1+\sqrt{1+|z|^4-|z|^2(2-|x-y|^2)}}{|x-y||z|^2},
  \]
  where $z$ is the hyperbolic midpoint of $x$ and $y$.
\end{lemma}
\begin{proof}
  By \eqref{hypeucl}
  \[
    t=\frac{(1-|z|^2)T}{1-|z|^2 T^2},
  \]
  where $T=\tanh ( \rho(x,y)/2 )$. Since $|x-y| \le 2t$ we obtain
  \[
    \frac{|z|^2-1+\sqrt{1+|z|^4-|z|^2(2-|x-y|^2)}}{|x-y||z|^2} \le T
  \]
  which implies the assertion.
\end{proof}

\begin{lemma}% Example 2010-02-19
  For $x,y \in \Bn$ we have
  \[
    \sinh \frac{\rho(x,y)}{2} \ge \frac{|x-y|}{2 \sqrt{\sqrt{1+r^2}\sqrt{r^2-\delta^2}-r^2}},
  \]
  where $\delta = |x-y|/2$, $r = |x-a|$ and $a$ is the center of the hyperbolic geodesic joining $x$ and $y$.
\end{lemma}
\begin{proof}
  Denote hyperbolic line through $x$ and $y$ by $C \subset S^1 (a,r)$. Consider points $x',y' \in C$ such that $\measuredangle (x,a,y) = \measuredangle (x',a,y')$ and $\measuredangle (0,a,y') = \measuredangle (x',a,y')/2$. Now
  \[
    |x'|^2 = |y'|^2 = \delta^2 + (\sqrt{1+r^2}-\sqrt{r^2-\delta^2})^2
  \]
  and thus
  \begin{eqnarray*}
    \sinh^2 \frac{\rho(x,y)}{2} & \ge &  \sinh^2 \frac{\rho(x',y')}{2} = \frac{|x'-y'|^2}{(1-|x'|^2)(1-|y'|^2)}\\
    & = & \frac{|x-y|^2}{4(\sqrt{1+r^2}\sqrt{r^2-\delta^2}-r^2)}.
  \end{eqnarray*}
\end{proof}

%COMPARE LOWER BOUNDS!!!

%LEMMA 4.3 = LEMMA 4.4!!!

%%%%%%%%%%%%%%%%%%%%%%%%%%%%%%%%%%%%%%%%%%%%%%%%
\section{Hyperbolic geometry in the half plane}

For convenient reference we record a well-known formula  which provides a connection
between the Euclidean and hyperbolic balls of $\H^n$ (see, e.g., \cite[(2.11)]{vu}) 
\begin{equation}\label{Brho}
  B_\rho(t e_n,r) = B^n(t e_n \cosh r,t \sinh r) \, \textrm{for} \quad r,t>0\,.
\end{equation}

The following result gives a formula for the hyperbolic geodesic segment in $\H^2$.

\begin{lemma}\label{geodesicH2}
  Let $x,y \in \H^2$ with $x_1 \neq y_1$. Then $S^1(c,r_c)$ is orthogonal to $\partial \H^2$, where
  \[
    c = \frac{|x|^2-|y|^2}{2(x_1-y_1)} \quad \textrm{and} \quad r_c = \sqrt{x_2^2+ \left( \frac{(x_1-y_1)^2+y_2^2-x_2^2}{2(x_1-y_1)} \right)^2 }.
  \]
\end{lemma}

\begin{lemma}
  Let $x,y \in \H^2$ with $x_2 < y_2$. Then
  \begin{eqnarray*}
    C & = & \{ z \in \R^2 \colon \rho(x,z)=\rho(y,z) \}\\
    & = & \left\{ z \in \R^2 \colon |x-z|=A|y-z|,\, A=\sqrt{x_2/y_2} \right\} = S^1(a,r_a),
  \end{eqnarray*}
  where
  \[
    a = \frac{x-A^2y}{1-A^2} \quad \textrm{and} \quad  r_a = \frac{A|x-y|}{1-A^2}.
  \]
  Moreover, the hyperbolic geodesic segment which contains $x $ and $y$ is orthogonal to $C$ and $c \in \partial \H^2$.
\end{lemma}
%\begin{proof}
%  Since $x_2 < y_2$ we have $A \in (0,1)$ and $C$ is a circle with center point on the line $l$ that contains $x$ and $y$. We denote intersection points of $l$ and $C$ by $z_1$ and $z_2$ so that $z_1 \in [x,y]$.Now $|x-y| = |x-z_1|+|y-z_1| = (1+1/A)|x-z_1|$ which is equivalent to $|x-z_1| = A|x-y|/(1+A)$ implying
%  \[
%    z_1 = x+(y-x)|x-z_1|/|x-y| = \frac{1}{1+A}x+\frac{A}{1+A}y = \frac{x+Ay}{1+A}.
%  \]
%  Similarly $|x-y| = |y-z_2|-|x-z_2| = (1/A-1)|x-z_2|$ which is equivalent to $|x-z_2| = A|x-y|/(1-A)$ implying
%  \[
%    z_2 = y+(x-y)(|x-y|+|x-z_2|)/|x-y| = \frac{1}{1-A}x-\frac{A}{1-A} = \frac{x-Ay}{1-A}.
%  \]
%  We obtain
%  \[
%    a = \frac{z_1+z_2}{2} = \frac{1}{2} \left( \frac{x+Ay}{1+A}+\frac{x-Ay}{1-A} \right) = \frac{x-A^2y}{1-A^2}
%  \]
%  and
%  \[
%    r_a = \frac{|x-z_1|+|x-z_2|}{2} = \frac{1}{2} \left( \frac{A|x-y|}{(1+A)} + \frac{A|x-y|}{(1-A)} \right) = \frac{A|x-y|}{1-A^2}.
%  \]
%
%  Since
%  \[
%    x-A^2y-(\overline{x}-A^2 \overline{y}) = 2i(x_2-A^2y_2) = 0
%  \]
%  we have $a \in \partial \H^2$ and the assertion follows.
%
%  To prove the orthogonality of the geodesic segment and $C$ it is sufficient to show that $r_a^2 +r_c^2=|c-a|^2$, where $c$ and $r_c$ as in Lemma \ref{geodesicH2}. This is a straightforward computation, which we omit.
%\end{proof}
\begin{proof}
  The assertion can easily be obtained by Lemma \ref{apollonian}, Corollary \ref{circlea2} and M\"obius trasformation $f(\Bn) = \Hn$.
\end{proof}

The following result gives the hyperbolic midpoint of two points $x$ and $y$.

\begin{lemma}\label{diamballH}
  For $x,y \in \H^2$ the smallest possible hyperbolic sphere that contains $x$ and $y$ is $B_\rho(z,\rho(x,y)/2)$ for
  \[
    z = \left( \frac{x_1 y_2+x_2 y_1}{x_2+y_2},\frac{\sqrt{x_2 y_2}\sqrt{(x_2+y_2)^2+(x_1-y_1)^2}}{x_2+y_2} \right).
  \]
\end{lemma}
\begin{proof}
  The geodesic that contains $x$ and $y$ is $S^1(c,r) \cap \H^2$ for
  \[
    c=\frac{x_1^2+x_2^2-y_1^2-y_2^2}{2(x_1-y_1)} \quad \textrm{and} \quad r=\sqrt{x_2^2+ \left( \frac{(x_1-y_1)^2+y_2^2-x_2^2}{2(x_1-y_1)} \right)^2 }.
  \]
  Since $|z-c|^2 = r^2$ we obtain $z_2 = \sqrt{r^2-(c-z_1)^2}$. Because $\rho(x,z) = \rho(z,y)$, which is equivalent to $\sqrt{x_2}|y-z|=\sqrt{y_2}|x-z|$, we obtain
  \[
    z_1 = \frac{x_1 y_2+x_2 y_1}{x_2+y_2}
  \]
  and the assertion follows.
\end{proof}

\begin{figure}[ht!]
  \begin{center}
    \includegraphics[width=10cm]{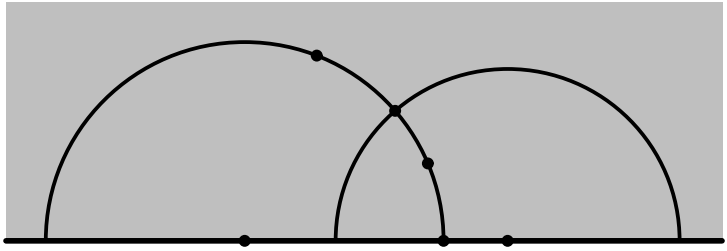}
    \caption{The hyperbolic midpoint in the upper half plane.}
  \end{center}
\end{figure}

%%%%%%%%%%%%%%%%%%%%%%%%%%%%%%%%%%%%%%%%%%%%%%%%
\section{Hyperbolic distance in the half space}

We prove lower bounds for the hyperbolic distance in $\Hn$.

\begin{lemma}\label{rhoH1}
  For $x,y \in \Hn$ we have
  \[
    \cosh \rho(x,y) \ge 1+\frac{|x-y|^2}{x_n^2+y_n^2}.
  \]
\end{lemma}
\begin{proof}
  The assertion follows since $2x_ny_n \le x_n^2+y_n^2$.
\end{proof}

%The next result is based on the following idea. For any $x,y \in \Hn$ denote the hyperbolic geodesic by $J$. We choose $x',y' \in J$ such that $|x'-y'| = |x-y|$ and $(x')_n = (y')_n$. Now
%\[
%  \rho(x',y') = \arccosh \left( 1+\frac{|x-y|^2}{2 (x')_n^2} \right)
%\]
%and $(x')_n^2 = |x-y|^2 (x_n+y_n)^2/(4|x'-y'|^2)$. The following lemma shows that $\rho(x',y') \le \rho(x,y)$.

\begin{lemma}\label{rhoH2}
  For $x,y \in \Hn$ we have
  \[
    \cosh \rho(x,y) \ge 1+\frac{2|x'-y'|^2}{(x_n+y_n)^2},
  \]
  where $x' = x-e_nx_n$ and $y' = y-e_ny_n$.
\end{lemma}
\begin{proof}
  We need to show that
  \[
    \frac{2|x'-y'|^2}{(x_n+y_n)^2} \le \frac{|x-y|^2}{2 x_n y_n},
  \]
  which is equivalent to $(x_n-y_n)^2(|x'-y'|^2+(x_n+y_n)^2) \ge 0$ and the assertion follows.
\end{proof}

It is natural to ask which one of the above lemmas gives better lower bound for $\rho(x,y)$. We need to find out when the inequality
\begin{equation}\label{comparison}
\frac{|x-y|^2}{x_n^2+y_n^2} \le \frac{2|x'-y'|^2}{(x_n+y_n)^2}
\end{equation}
holds. Since \eqref{comparison} is equivalent to $x_n+y_n \le |x'-y'|$ we obtain that the lower bound of Lemma \ref{rhoH2} is better than the lower bound of Lemma \ref{rhoH1} whenever $x_n+y_n \le |x'-y'|$.

\begin{lemma}
  For $x,y \in \Hn$ we have
  \begin{eqnarray*}
    \sinh \frac{\rho(x,y)}{2} & \ge & \frac{x_n+y_n}{2\sqrt{x_n y_n}} \sqrt{1-\frac{4x_n y_n}{(x_n+y_n)^2+|x'-y'|^2}}\\
    & \ge & \frac{(x_n+y_n)|x'-y'|}{2\sqrt{x_n y_n}{\sqrt{(x_n+y_n)^2+|x'-y'|^2}}},
  \end{eqnarray*}
  where $x' = x-e_nx_n$ and $y' = y-e_ny_n$.
\end{lemma}
\begin{proof}
  By \eqref{Brho} we have $B_\rho(t e_n,r) = B^n((t \cosh r)e_n,t \sinh r)$ implying $|x-y| \le 2t \sinh r$ for all $x,y \in \partial B_\rho(t e_n,r)$. Therefore, by Lemma \ref{diamballH} we have for all $x,y \in \Hn$ that
  \[
    |x-y| \le 2 \frac{\sqrt{x_n y_n}\sqrt{(x_n+y_n)^2+|x'-y'|^2}}{x_n+y_n}\sinh \frac{\rho(x,y)}{2}
  \]
  which is equivalent to
  \begin{eqnarray*}
    \rho(x,y) & \ge & 2\arcsinh \left( \frac{x_n+y_n}{2\sqrt{x_n y_n}} \sqrt{\frac{(x_n-y_n)^2+|x'-y'|^2}{(x_n+y_n)^2+|x'-y'|^2}} \right)\\
    & = & 2\arcsinh \left( \frac{x_n+y_n}{2\sqrt{x_n y_n}} \sqrt{1-\frac{4x_n y_n}{(x_n+y_n)^2+|x'-y'|^2}} \right)\\
    & \ge & 2\arcsinh \frac{(x_n+y_n)|x'-y'|}{2\sqrt{x_n y_n}{\sqrt{(x_n+y_n)^2+|x'-y'|^2}}}
  \end{eqnarray*}
  and the assertion follows.
\end{proof}

%COMPARE LOWER BOUNDS!!!

{\small
%%%%%%%%%%%%%%%%%%%%%%%%%%%%%%%%%%%%%%%%%%%%%%%%

} %end small

\bigskip
\noindent
{\sc
Riku Kl\'en and Matti Vuorinen}\\
Department of Mathematics and Statistics\\
University of Turku\\
20014 Turku\\
Finland\\
ripekl@utu.fi,   vuorinen@utu.fi\\

\end{document}